\documentclass{article}
\usepackage{graphicx} 
\usepackage{hyperref}   
\hypersetup{
	colorlinks=true,
	linkcolor=blue,
	citecolor=blue,
	urlcolor=blue,
}
\usepackage{amsthm}
\usepackage{xcolor}
\usepackage{algorithm}
\usepackage{algpseudocode}
\usepackage{graphicx}
\usepackage{amsfonts}
\usepackage{amsmath,mathtools,amssymb,amsfonts}

\newcommand{\norm}[1]{\left\lVert#1\right\rVert}
\newcommand{\xddots}{%
  \raise 4pt \hbox {.}
  \mkern 6mu
  \raise 1pt \hbox {.}
  \mkern 6mu
  \raise -2pt \hbox {.}
}
\DeclareMathOperator*{\SubjectTo}{Subject\phantom{a}to:}
\DeclareMathOperator*{\Minimize}{Minimize:}
\DeclareMathOperator*{\Maximize}{Maximize:}
\DeclareMathOperator*{\argmin}{arg\,min}

\newtheorem{assumption}{\it{Assumption}}
\newtheorem{theorem}{\bf{Theorem}}

\newtheorem{proposition}{\bf{Proposition}}
\newtheorem{remark}{\bf{Remark}}

\title{Fault-Tolerant Decentralized Control for Large-Scale Inverter-Based Resources for Active Power Tracking}
\author{Satish Vedula, Ayobami Olajube, and \\ Olugbenga Moses Anubi}

\date{}

\begin{document}

\maketitle
\begin{centering}
Department of Electrical and Computer Engineering, the Center for Advanced Power Systems, Florida State University \\
E-mail: \{svedula, aolajube, oanubi\}@fsu.edu
\end{centering}

\section{Abstract}
Integration of Inverter Based Resources (IBRs) which lack the intrinsic characteristics such as the inertial response of the traditional synchronous-generator (SG) based sources presents a new challenge in the form of analyzing the grid stability under their presence. While the dynamic composition of IBRs differs from that of the SGs, the control objective remains similar in terms of tracking the desired active power. This letter presents a decentralized primal-dual-based fault-tolerant control framework for the power allocation in IBRs. Overall, a hierarchical control algorithm is developed with a lower level addressing the current control and the parameter estimation for the IBRs and the higher level acting as the reference power generator to the low level based on the desired active power profile. The decentralized network-based algorithm adaptively splits the desired power between the IBRs taking into consideration the health of the IBRs transmission lines. The proposed framework is tested through a simulation on the network of IBRs and the high-level controller performance is compared against the existing framework in the literature. The proposed algorithm shows significant performance improvement in the magnitude of power deviation and settling time to the nominal value under faulty conditions as compared to the algorithm in the literature.

\section{Introduction}
Electrical Power Systems (EPS) are undergoing substantial structural and operational transformation by replacing the existing Synchronous Machines (SMs) with the Inverter Based Resources (IBRs). In contrast to the SMs, the IBRs exhibit low inertia, resulting in a swift response to stochastic events. However, this presents considerable control challenges related to the stability and robustness of the EPS. In the existing literature, several researchers have proposed methods for controlling large-scale IBRs such as discrete-time consensus control using proportional derivative PD \cite{chen2020distributed}, and distributed droop control \cite{schiffer2015voltage}. However, these approaches lack robustness. A distributed model predictive control (MPC) for droop-controlled IBRs was proposed by \cite{anderson2019distributed}. Since this approach is based on MPC, it is computationally heavy. A distributed sliding mode control (SMC) for islanded AC microgrids was presented in ~\cite{alfaro2021distributed}. SMC-based control techniques show great robustness against exogenous disturbances but surfers from chattering effects. Also, observers are often required for disturbance estimation and chattering reduction which may increase with switching gain in inverters thereby increasing the complexity of implementing SMC \cite{9829024}.

Moreover, in a power system network proliferated with IBRs, there is a possibility of losing any of the distributed energy resources (DERs) because of poor response to time-varying load changes or faults. Thus, recovering the aggregated total output power from collective IBRs may be an issue. This presents a serious control problem. In \cite{AMELI}, a robust adaptive control technique is used to track the aggregate output power of the IBRs when one of them is lost. The result gave a perfect output power tracking that matches the total power contributed by each unit of IBRs.  A dispatchable virtual oscillator control is designed for a network of IBRs to track the desired active power, and voltage magnitude in~\cite{subotic2020lyapunov}. The method provides sufficient conditions for voltage stability but does not give the admissible set for desired powers. 

A renowned method for maintaining power sharing in the power grids dominated by IBRs is the droop control technique. The control algorithm is saddled with voltage and frequency deviation issues emanating from uncertainties in the output impedance and poor transient performance. Traditional droop control is very effective in systems with resistive output impedance which results in poor grid stability \cite{en15124439}. In as much as droop control techniques have been modified to tackle these issues, tracking the cumulative output power resulting from the loss of an inverter or the poor output impedance condition is still an open problem. Since the contemporary power droop control method degrades with the line impedance, modified droop control has been developed to improve the active power sharing among the grid-connected inverters \cite{8660486}. 

Nevertheless, the certificate of stability in the presence of the disturbances and the active power tracking is not guaranteed to employ the droop methods. Few authors have proposed a robust control-based approach to address stability and power quality issues arising from the integration of IBRs \cite{Anubi, faiz2020h}. An optimization-based approach was proposed in \cite{Bidram}. However, the active power tracking under IBR failure was not discussed. In \cite{CHANG201685}, the authors proposed a distributed control for IBR coordination in the islanded microgrids. However, the impact of the failure of the IBR was not discussed. Thus, this paper focuses on the tracking of the aggregated active power under IBR failure in a decentralized fault-tolerant framework. The main contributions in the paper are:
\begin{enumerate}
    \item A decentralized fault-tolerant high-level control is proposed that adaptively splits the active power among the IBRs in the presence of faults.
    \item An adaptive estimator for the parameter estimation is designed with the stability analysis. 
    \item A faster response in power shared during faulty conditions is demonstrated compared to the previously proposed approach.
    \item Improvement in the active power tracking under faults is demonstrated compared to previously proposed work.
\end{enumerate}

The paper is organized as follows: Section-\ref{Not_Pre} introduces the mathematical notations used throughout the paper and the preliminaries to better understand the control development. In Section-\ref{Model}, the inverter model is presented followed by the low-level and high-level control design in Section-\ref{Control}. Section-\ref{Sim} presents the simulation results of the designed controller and the results are compared with the results in the literature.

\section{Notations and Preliminaries}\label{Not_Pre}
$\mathbb{N}$, $\mathbb{R}$, and $\mathbb{R}_+$ denote the set of natural, real, and positive real numbers. $\mathbb{L}_2$ and $\mathbb{L}_{\infty}$ denote the square-integrable (measurable) and bounded signal spaces. A real matrix with $n$ rows and $m$ columns is denoted as $X \in \mathbb{R}^{n \times m}$. The identity matrix is denoted as $I$. $X^\top$ denotes the transpose of a matrix $X$. Natural and Real scalars are denoted by lowercase alphabets (for example $x \in \mathbb{N}$ and $y \in \mathbb{R}$). The real vectors are represented by the lowercase bold alphabets (i.e. $\textbf{x} \in \mathbb{R}^{n}$). The vector of ones and zeros is denoted as $\mathbf{1}$ and $\underline{\mathbf{0}}$. 
 For any vector $\mathbf{x} \in \mathbb{R}^n$, $\|\mathbf{x}\|_2 \triangleq \sqrt{\mathbf{x}^\top\mathbf{x}}$ and $\|\mathbf{x}\|_1
 \triangleq \sum_{i=1}^{n}|\mathbf{x}_i|$, representing the 2-norm and the 1-norm, respectively (where $|.|$ denotes absolute value). The symbol $\preceq$ denotes the component-wise inequality i.e. $\mathbf{x} \preceq \mathbf{y}$ is equivalent to $\mathbf{x}_i \leq \mathbf{y}_i$ for $i=1,2,\hdots,n$. The dot product/inner product of the two vectors $\mathbf{x} \in \mathbb{R}^n$ and $\mathbf{y} \in \mathbb{R}^n$ is denoted as $\mathbf{x}^\top \mathbf{y}$. The cross product/outer product for $n=2$ is denoted as $\mathbf{x}^\top J \mathbf{y}$, where $J \triangleq\begin{bmatrix}
     0 & 1 \\ -1 & 0
 \end{bmatrix}$. For a function $f:\mathbb{R}^n \longrightarrow \mathbb{R}^m$, $\nabla f(x)$ denotes the gradient of the function $f$ at $x$. $\nabla^2f(x)$ denotes the hessian of the function $f$ at $x$. $\mathcal{B}(0,r)\triangleq\left\{\mathbf{x}\in\mathbb{R}^n|\left\|\mathbf{x}\right\|<r\right\}\subset \mathbb{R}^n$ denotes a ball of radius $r$ centered at the origin. 

\subsection{Preliminaries}  
    Consider the network with $m \in \mathbb{N}$ nodes or agents labeled by the set $\mathcal{V}=\{1,2,\hdots,m\}$, $\mathcal{E}$ is the \emph{unordered} edge such that $\mathcal{E} \subseteq \mathcal{V} \times \mathcal{V}$. The connection between the nodes is fixed $\mathcal{G}=(\mathcal{V},\mathcal{E})$. For an undirected graph, The \emph{adjacency} matrix ($A \in \mathbb{R}^{m \times m}$) is denoted as 
    $a_{ij}= \{1 \hspace{1mm}  \text{if} \hspace{1mm} (i,j)\in\mathcal{E} | 0 \hspace{1mm}  \text{otherwise}\},$ 
    i.e. if there is a path between two nodes, then the value of 1 is assigned otherwise the value in the matrix is set to 0. The \emph{degree} of a graph ($\mathcal{N}_i$) denotes the number of neighboring nodes of a given node. The degree matrix ($D \in \mathbb{R}^{m \times m}$) is the degree of a given node on the diagonal and 0's elsewhere. The \emph{Laplacian} matrix ($L=D-A$), is the difference between the degree and the adjacency matrix. The Laplacian-based weighted graph matrix is defined as follows: $W = I-\frac{1}{\tau}L$, where $\tau > \frac{1}{2}\lambda_{max}L$ is a constant, $\lambda_{max}$ is the maximum eigenvalue of $L$.

\begin{figure}[h!] 
\centerline{\includegraphics[width=2.8 in]{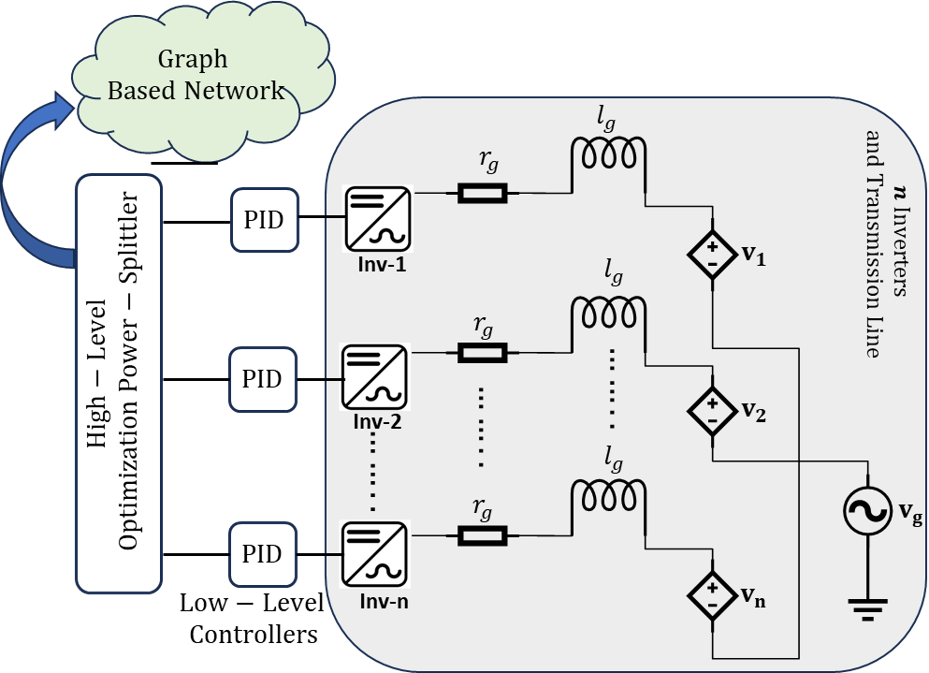}}
\caption{Hierarchical optimization based control design for IBRs}
\label{fig_example5.jpg}
\end{figure}

\section{Model Development}\label{Model}
The dynamics of a \emph{single} grid following IBR model connected to the grid via a transmission line in $dq$ coordinates is given as follows \cite{levron}:
\begin{equation}\label{Inverter_Model}
    l_g\frac{d\mathbf{i}}{dt} = -(r_g I-l_g \omega_g J)\mathbf{i}(t) + \mathbf{v}(t) -\mathbf{v}_g,
\end{equation}
where, $\mathbf{i}(t) \in \mathbb{R}^2$, $\mathbf{v}(t) \in \mathbb{R}^2$ and $\mathbf{v}_g \in \mathbb{L}_2$ are the inverter current, voltage (control input) and measured grid voltage respectively. $r_g \in \mathbb{R}_+$ and $l_g \in \mathbb{R}_+$ is the resistance and the inductance of the transmission line connecting to the grid in \textsf{Ohm} and \textsf{Henry}. $\omega_g \in \mathbb{R}_+$ is the grid frequency in \textsf{rad/s}. 
\begin{assumption} For simplicity of exposition, the grid frequency $\omega_g$ is considered to be fixed. Also, the switching dynamics of the inverter are neglected \cite{Anubi}. The grid voltage $\mathbf{v}_g$ is regulated to a fixed value since the IBRs are designed in a grid-following mode.
\end{assumption}

\section{Control Development}\label{Control}
The overall control development model consists of the high-level aggregated power tracking controller acting as a reference current generator for the low-level controller as shown in Figure \ref{fig_example5.jpg}. 
\subsection{Low-Level Control}
The objective of the low-level control design is to track a given reference power signal asymptotically while learning the fault indicating uncertain model parameters. Since the voltage is assumed to be regulated, the power tracking problem is converted to a current tracking problem. Namely, given desired active and reactive power pair $(p_{ref},q_{ref})$, the desired reference current is given by
\begin{equation}
    \mathbf{i}_{ref} = \frac{\mathbf{v}_g}{\norm{\mathbf{v}_g}_2^2}p_{ref}+J\frac{\mathbf{v}_g}{\norm{\mathbf{v}_g}_2^2}q_{ref}.
\end{equation}
To achieve the adaptive current tracking objective, we consider the \emph{fault-free} reference model 
   \begin{equation}\label{Model Reference}
   l_m \frac{d\mathbf{i}_m}{dt}=-\bigg((r_m+k)I-l_m\omega_gJ\bigg)\left(\mathbf{i}_m(t)-\mathbf{i}_{ref}(t)\right),
   \end{equation}  
   where $\mathbf{i}_m(t) \in \mathbb{R}^{2}$, $\mathbf{i}_{ref}(t) \in \mathbb{R}^{2}$ is the desired and bounded reference, $r_m >0, l_m >0$ are the known model reference parameters, $k>0$ is a convergence rate tuning parameter. Consequently, the model reference active and reactive powers, $p_m\triangleq\mathbf{v}_g^\top\mathbf{i}_m$ and $q_m\triangleq\mathbf{v}_g^\top J\mathbf{i}_m$, satisfy
   \begin{align}
       \left\|\left[\begin{array}{c}p_m(t)-p_{ref}\\q_m(t)-q_{ref}\end{array}\right]\right\|_2 \le \left\|\left[\begin{array}{c}p_m(0)-p_{ref}\\q_m(0)-q_{ref}\end{array}\right]\right\|_2e^{-\frac{r_m+k}{l_m}t},
   \end{align}
which shows an exponential tracking of the power references with a rate that can be tuned by the control gain $k$.

Next, consider the current tracking error, between the system dynamics in \eqref{Inverter_Model} and the model reference in \eqref{Model Reference}, as follows:
\begin{equation}\label{measurement_tracking}
    \tilde{\mathbf{i}}(t) = \mathbf{i}(t)-\mathbf{i}_m(t).
\end{equation}
Taking the first time derivative yields the error dynamics
\begin{equation}
\begin{aligned}
    l_g\dot{\tilde{\mathbf{i}}}(t) = -(r_gI-l_g\omega_gJ)\mathbf{i}(t)+\mathbf{v}(t)-\mathbf{v}_g+((r_m+k)I\\-l_m\omega_gJ)\mathbf{i}_m(t)-((r_m+k)I-l_m\omega_gJ)\mathbf{i}_{ref}(t), 
\end{aligned}
\end{equation}
adding and subtracting $\hat{r}_g\mathbf{i}(t)$ and setting that the value of the inductance $l_g = l_{g_0} + \Delta l_g$, where $\hat{r}_g$ is an estimate of $r_g$ to be designed and $l_{g_0}$ is the known nominal inductance value with $\Delta l_g$ the associated parametric uncertainty, yields:
\begin{align*}
\l_g\dot{\tilde{\mathbf{i}}}(t) &= -(r_m+k)\tilde{\mathbf{i}}(t)+\tilde{r}_g\mathbf{i}(t)-(\hat{r}_g-r_m-k)\mathbf{i}(t)+\\&\hspace{10mm}(l_{g_0}+\Delta l_g)\omega_g J \mathbf{i}(t)-l_m\omega_g J\mathbf{i}_m(t)-\mathbf{v}_g-\\&\hspace{16mm}((r_m+k)I-l_m\omega_gJ)\mathbf{i}_{ref}(t)+\mathbf{v}(t),
\end{align*}
where $\tilde{r}_g=\hat{r}_g-{r}_g$. Consider the control law
\begin{equation}
\begin{aligned}
   \mathbf{v}(t)=(\hat{r}_g-r_m-k)\mathbf{i}(t)-l_{g_0} \omega_g J \mathbf{i}(t)+l_m \omega_g J \mathbf{i}_m(t)\\+ ((r_m+k)I-l_m\omega_gJ)\mathbf{i}_{ref}(t) + \mathbf{v}_g.
\end{aligned}
\end{equation}
Thus, the error system becomes
\begin{equation}\label{closed_loop_error}
   l_g\dot{\tilde{\mathbf{i}}}(t) = -(r_m+k) \tilde{\mathbf{i}}(t)+\tilde{r}_g\mathbf{i}(t) + \Delta l_g\omega_g J \mathbf{i}(t), 
\end{equation}
with the associated update laws given in the next result
\begin{theorem}
   Given $\varepsilon>0$ and $\bar{r}>0$. Consider the error dynamics in (\ref{closed_loop_error}). If the parameter estimates satisfy the update law
   \begin{equation}\label{update_laws}
   \dot{\hat{r}}_g=\gamma_r \textsf{Proj}_f\left(\hat{r}_g,-\tilde{\mathbf{i}}(t)^\top\mathbf{i}(t)\right),\hspace{2mm}|\hat{r}_g(0)|<\bar{r},
   \end{equation}
   where $\gamma_r > 0$ is the associated adaptation rate for the parameter $r_g$ and $f(r) = \frac{r^2-\bar{r}^2}{2\varepsilon\bar{r}+\varepsilon^2}$, then the origin of the closed-loop system in (\ref{closed_loop_error}) is globally asymptotically stable.  
\end{theorem}
\begin{proof}
    We drop the explicit usage of time for the simplicity of writing the proof. Consider the following Lyapunov candidate function
    $$V(\tilde{\mathbf{i}}, \tilde{r}_g) = \frac{l_g}{2}\tilde{\mathbf{i}}^\top  \tilde{\mathbf{i}} + \frac{1}{2\gamma_r}\tilde{r}_g^2,$$
    taking the first derivative along the time-variables and substituting (\ref{closed_loop_error}) yields
\begin{align*}
    \dot{V} &= \frac{1}{2}\dot{\tilde{\mathbf{i}}}^\top l_g \tilde{\mathbf{i}}+\frac{1}{2}\tilde{\mathbf{i}}^\top l_g \dot{\tilde{\mathbf{i}}}+\frac{1}{\gamma_r}\tilde{r}_g\dot{\hat{r}}_g,\\
    &= -\frac{(r_m+k)}{2}\tilde{\mathbf{i}}^\top\tilde{\mathbf{i}}+\frac{\tilde{r}_g}{2}\mathbf{i}^\top\tilde{\mathbf{i}}+\frac{1}{2}\Delta l_g \omega_g \mathbf{i}^\top J^\top \tilde{\mathbf{i}}\\&\hspace{10mm}-\frac{(r_m+k)}{2}\tilde{\mathbf{i}}^\top \tilde{\mathbf{i}}+\frac{\tilde{r}_g}{2}\tilde{\mathbf{i}}^\top \mathbf{i}+\frac{1}{2}\Delta l_g \omega_g \tilde{\mathbf{i}}^\top J \mathbf{i}\\&\hspace{5cm}+\frac{1}{\gamma_r}\tilde{r}_g\dot{\hat{r}}_g,\\
    &= -(r_m+k)\tilde{\mathbf{i}}^\top\tilde{\mathbf{i}} +\tilde{r}_g\tilde{\mathbf{i}}^\top \mathbf{i}+\\&\hspace{20mm}\Delta l_g \omega_g \mathbf{i}^\top(J+J^\top)\tilde{\mathbf{i}}+\frac{1}{\gamma_r}\tilde{r}_g\dot{\hat{r}}_g,
\end{align*}
using the skew-symmetric property  $(J^\top = -J)$ and substituting the update law (\ref{update_laws}), and using the property $\tilde{r}_g\left(\textsf{Proj}_f(\hat{r}_g,-y)+y\right)\le0 \text{ for all } y\in\mathbb{R}$, yields
\begin{equation}\label{Lyapunov_derivative}
\begin{aligned}
    \dot{V} &= -(r_m+k) \norm{\tilde{\mathbf{i}}}^2 .
\end{aligned}    
\end{equation}
$\dot{V}(t)$ is negative semi-definite (NSD) and $V(t)>0$. Thus, $V(t) \in \mathbb{L}_{\infty}$ which implies $\tilde{\mathbf{i}}(t), \tilde{r}_g(t) \in \mathbb{L}_{\infty}$. Since $\mathbf{i}_m(t)$ is assumed to be bounded, it implies that $\mathbf{i}(t) \in \mathbb{L}_{\infty}$. Consequently, the control input $\mathbf{v}(t) \in \mathbb{L}_{\infty}$. Integrating (\ref{Lyapunov_derivative})
\begin{equation}
    \begin{aligned}
        V(\infty)-V(0) \leq -(r_m+k) \int_{0}^{\infty}\norm{\tilde{\mathbf{i}}(t)}^2 dt,
    \end{aligned}
\end{equation}
It follows that $\tilde{\mathbf{i}}(t) \in \mathbb{L}_2$. From the implications $\tilde{\mathbf{i}}(t)$ is uniformly continuous. Thus, invoking Barbalat's lemma \cite{khalil2002nonlinear} it follows that $\tilde{\mathbf{i}}(t)\longrightarrow \underline{\mathbf{0}}$.
\end{proof}
\begin{proposition}\label{prop-1}
    The current tracking error $\tilde{\mathbf{i}}(t)$ is uniformly bounded according to
    \begin{equation}\label{bound_conv}
        \sup_{t\in\mathbb{R}_+}\norm{\tilde{\mathbf{i}}(t)}_2 \leq \norm{\tilde{\mathbf{i}}(0)}_2 + \frac{\overline{r}}{\sqrt{2\gamma_r}} ,
    \end{equation}
\end{proposition}
\begin{proof}
    Consider (\ref{Lyapunov_derivative}), without loss of generality it can be expressed as
    \begin{align*}
        \dot{V} &= -(r_m+k) \norm{\tilde{\mathbf{i}}}^2, \\
         &\leq \frac{-2k}{l_g}V + \frac{k}{l_g \gamma_r}\tilde{r}_g^2
         \leq \frac{-2k}{l_g}V + \frac{k}{l_g \gamma_r}\overline{r}^2,
    \end{align*}
    using the comparison lemma (\cite{khalil2002nonlinear}), it follows that
    \begin{align*}
        V(t) &\leq e^{\frac{-2k}{l_g}t}V(0)+\frac{\overline{r}^2}{2\gamma_r}\bigg(1-e^{\frac{-2k}{l_g}t}\bigg), \\
        \norm{\tilde{\mathbf{i}}(t)}^2 &\leq e^{\frac{-2k}{l_g}t}\norm{\tilde{\mathbf{i}}(0)}^2+\frac{\overline{r}^2}{2\gamma_r}\bigg(1-e^{\frac{-2k}{l_g}t}\bigg),
    \end{align*}
Thus, $\mathbf{i}(t) \in \mathcal{B}\bigg(\mathbf{i}_m(t),\norm{\tilde{\mathbf{i}}(0)} + \frac{\overline{r}}{\sqrt{2\gamma_r}}\bigg)$. 
\end{proof}

From the Proposition \ref{prop-1}, the resulting power from the \emph{single} IBR satisfies

\begin{align*}
    p(t) &= \mathbf{v}_g(t)^\top \mathbf{i}(t) \in \mathcal{B}\bigg(p_{ref},\norm{\mathbf{v}_g}\left(\norm{\tilde{\mathbf{i}}(0)}+\frac{\overline{r}}{\sqrt{2\gamma_r}}\right)\bigg), \\
    q(t) &= \mathbf{v}_g(t)^\top J \mathbf{i}(t) \in \mathcal{B}\bigg(q_{ref},\norm{\mathbf{v}_g}\left(\norm{\tilde{\mathbf{i}}(0)}+\frac{\overline{r}}{\sqrt{2\gamma_r}}\right)\bigg).
\end{align*}

The total active power supplied by the $n$ number of IBRs satisfies
\begin{equation*}
   \sum_{i=1}^{n} p_i(t)  \in \mathcal{B}\Bigg(\sum_{i=1}^{n} p_{{ref}_i},\norm{\mathbf{v}_g}\sum_{i=1}^{n}\bigg(\norm{\tilde{\mathbf{i}}_{i(0)}}+\frac{\overline{r}_i}{\sqrt{2\gamma_r}_i}\bigg)\Bigg).
\end{equation*}

Moreover, from the update law in (\ref{update_laws}), it is seen that the parameter error is driven by the current tracking error. Thus, if the system is persistently excited \cite{khalil2002nonlinear}, the healthier systems (having small parametric deviations) will result in less error in the active power tracking. Thus, to minimize uncertainty in the total active power, we consider a resource allocation problem using the parameter deviation error to penalize the power requested from the local ($i^{th})$ IBR. 

\subsection{High-Level Control}
The objective of the high-level control is to track the total desired power by allocating more reference to healthier IBRs (those with less parameter deviation). This is achieved through the resource allocation problem:
\begin{align}\label{Main_Opt}
    \Minimize_{p_i} & \sum_{i=1}^{n}\frac{\beta_i(\tilde{r}_{g_i})}{2}f_i(p_i) \hspace{1mm} \SubjectTo  \sum_{i=1}^{n} p_i = p_A\\ \label{Reactive_opt}
    \Minimize_{q_i} & \sum_{i=1}^{n}\frac{\beta_i(\tilde{r}_{g_i})}{2}h_i(q_i) \hspace{1mm} \SubjectTo  \sum_{i=1}^{n} q_i = q_A ,
\end{align}
where, $n\in \mathbb{N}$ are the number of IBRs, $\beta_i(\tilde{r}_{g_i})$ is the line parameter dependent penalty associated with the $i^{th}$ IBR. It is determined based on the fault indication and the status of the transmission line. $f_i:\mathbb{R}\longrightarrow\mathbb{R}_+$ is the cost associated with the active power and $h_i: \mathbb{R}\longrightarrow\mathbb{R}_+$ is the cost associated with the reactive power of the individual IBR operation. It is assumed that the functions $f_i, h_i$ are $\mu$-strongly convex and $L$-smooth. $p_i \in \mathbb{R}, q_i \in \mathbb{R}$ are the active and the reactive power of the individual IBR. $p_A, q_A$ are the desired aggregated active and reactive power that needs to be tracked by the IBRs. 
\begin{remark}
        The optimization problems in (\ref{Main_Opt}) and  (\ref{Reactive_opt}) are equality-constrained problems and a closed-loop solution can be obtained offline. However, they depend on the parameter-based penalty parameter $\beta_i(\tilde{r}_{g_i})$ thus the computation needs to be performed online. Moreover, having a centralized controller acts as a point of failure which leads to instabilities in the grid in case the controller fails. Thus, we decentralize the controller to address the single point of failure.
\end{remark}

\textit{Decentralized Development:} The \textit{Lagrangian} for the optimization problems (\ref{Main_Opt}) and (\ref{Reactive_opt}) is given as follows:
\begin{subequations}
\begin{align}
    \mathcal{L}(p,\lambda) = \sum_{i=1}^{n}\frac{\beta_i(\tilde{r}_{g_i})}{2}f_i(p_i)+\lambda (\sum_{i=1}^{n}p_i-p_A),\\
    \mathcal{L}(q,\nu) = \sum_{i=1}^{n}\frac{\beta_i(\tilde{r}_{g_i})}{2}h_i(q_i)+\nu (\sum_{i=1}^{n}q_i-q_A),
\end{align}
\end{subequations}
where $\lambda, \nu \in \mathbb{R}$ are the dual variables for active and reactive powers. Given $\lambda$ and $\nu$, let $p_i^*$ and $q_i^*$ to be the solution of the optimization problem in (\ref{Main_Opt})-(\ref{Reactive_opt}) and is given as
\begin{subequations}
\begin{align}
    p_i^*(\lambda) = \argmin_{p_i} \bigg(\frac{\beta_i(\tilde{r}_{g_i})}{2}f_i(p_i) + \lambda p_i \bigg), \\
    q_i^*(\nu) = \argmin_{q_i} \bigg(\frac{\beta_i(\tilde{r}_{g_i})}{2}h_i(q_i) + \nu q_i \bigg).
\end{align}
\end{subequations} 

Thus, given $p_i^*$ and $q_i^*$, the \emph{Lagrangian dual functions} are defined as the minimum values of the Lagrangian over $p_i$ and $q_i$
\begin{subequations}
\begin{align}
    g(\lambda) \triangleq \inf_{p_i} \mathcal{L}(p_i,\lambda) \equiv \mathcal{L}(p_i^*,\lambda), \\
    d(\nu) \triangleq \inf_{q_i} \mathcal{L}(q_i,\nu) \equiv \mathcal{L}(q_i^*,\nu),
\end{align}
\end{subequations}
Consequently, the dual problems for the active and the reactive power are given as:
\begin{align}\label{Main_Dual}
    \Maximize_\lambda  g(\lambda) && \Maximize_\nu  d(\nu).
\end{align}
Assume the dual problem can be broken down into $n$ sub-problems.  (\ref{Main_Dual}) can be written as:
\begin{subequations}\label{Second_Dual}
\begin{align}
    \Maximize_{\lambda_i=\lambda_j} \quad \sum_{i=1}^{n} g_i(\lambda_i), \hspace{2mm} \forall \{i,j\} \in \mathcal{G} \\
    \Maximize_{\nu_i=\nu_j} \quad \sum_{i=1}^{n} d_i(\nu_i), \hspace{2mm} \forall \{i,j\} \in \mathcal{G}
\end{align}
\end{subequations}
the optimization problem in (\ref{Second_Dual}a-b) are solved over a network $\mathcal{G}$. It is assumed that the network has a \emph{spanning tree}. The information of the reference active and reactive powers $p_A, q_A$ is available only for the first node. Thus, the gradient update step for the $\lambda$ and $\nu$ is given as follows ($k$ here is iteration counter):
\begin{equation}
\lambda_i^{k+1} = \begin{cases}  \lambda_i^k + \alpha\sum_{j \in \mathcal{N}_i}w_{ij} (p_i^* - p_A) \hspace{3mm} \text{for} \hspace{1mm} i=1, \\
      \lambda_i^k + \alpha \sum_{j \in \mathcal{N}_i}w_{ij} p_i^* \hspace{14mm} \text{otherwise},
\end{cases}    
\end{equation}
\begin{equation}
\nu_i^{k+1} = \begin{cases}  \nu_i^k + \alpha \sum_{j \in \mathcal{N}_i}w_{ij} (q_i^* - q_A) \hspace{3mm} \text{for} \hspace{1mm} i=1, \\
      \nu_i^k + \alpha \sum_{j \in \mathcal{N}_i}w_{ij} q_i^* \hspace{14mm} \text{otherwise},
\end{cases}    
\end{equation}
where $\alpha > 0$ is the ascent step size assumed to be a constant for every update step. Consequently, the high-level power splitting decentralized optimization algorithm is given as the following iterative scheme:
\begin{subequations}\label{Network_Opt}
\begin{align}
    p_i^{k+1} &= \argmin_{p_i} \bigg(\frac{\beta_i(\tilde{r}_{g_i})}{2}f_i(p_i) + \lambda_i^k p_i \bigg)  \\
    \boldsymbol{\lambda}^{k+1}  &= \boldsymbol{\lambda}^k + \alpha W(\mathbf{p}^{t+1}-\mathbf{p}_A), \\
    q_i^{k+1} &= \argmin_{q_i} \bigg(\frac{\beta_i(\tilde{r}_{g_i})}{2}h_i(q_i) + \nu_i^k q_i \bigg)  \\
    \boldsymbol{\nu}^{k+1}  &= \boldsymbol{\nu}^k + \alpha W(\mathbf{q}^{k+1}-\mathbf{q}_A),
\end{align}    
\end{subequations}
where $p_i^{k+1} \in \mathbb{R}, q_i^{k+1} \in \mathbb{R}$ are the optimal active and reactive power of the individual IBR, $W \in \mathbb{R}^{n \times n}$ is the graph Laplacian based weigh matrix. $\mathbf{p}^{k+1} \triangleq \left[\begin{array}{cccc} p_1^{k+1}&p_2^{k+1}&\hdots&p_n^{k+1}\end{array}\right]^{\top} \in \mathbb{R}^n$ and \newline $\mathbf{q}^{k+1} \triangleq \left[\begin{array}{cccc} q_1^{k+1}&q_2^{k+1}&\hdots&q_n^{k+1}\end{array}\right]^{\top} \in \mathbb{R}^n$ are the vectors of primal optimal active and reactive power profiles. $\mathbf{p}_A \triangleq \left[\begin{array}{cccc} p_A&0&\hdots&0\end{array}\right]^{\top} \in \mathbb{R}^n$ and $\mathbf{q}_A \triangleq \left[\begin{array}{cccc} q_A&0&\hdots&0\end{array}\right]^{\top} \in \mathbb{R}^n$ are the reference active and reactive power profiles. $f_i$ and $h_i$ are assumed to be $\mu$-strongly convex and $L$-smooth. The primal optimization problems in (\ref{Network_Opt}a) and (\ref{Network_Opt}c) are unconstrained problems and under the assumption that the fixed points $\mathbf{p}^*$ and $\mathbf{q}^*$ exist and the gradient step is a \emph{contraction} the following linear convergence property holds
\begin{subequations}\label{convergence}
\begin{align}
    \norm{\mathbf{p}^{N}-\mathbf{p}^*} \leq L_p^{N-1} \norm{\mathbf{p}_0-\mathbf{p}^*} \\
    \norm{\mathbf{q}^{N}-\mathbf{q}^*} \leq L_q^{N-1} \norm{\mathbf{q}_0-\mathbf{q}^*}
\end{align}    
\end{subequations}
where $N$ is the number of iterations, $L_p, L_q <1 \triangleq max\{\left|1-\alpha \mu \beta\right|,\left|1-\alpha L \beta\right|\}$ are the contraction coefficients. $\mathbf{p}_0(t), \mathbf{q}_0(t)$ are the system active and reactive power measurements acting as the initial values for the optimization problem. The choice of step size $\alpha = 2/(L+\mu \beta)$ minimizes the contraction coefficient. Thus, the condition number defined as $\gamma = L/\mu \beta$ dictates the convergence of the optimization problem and the problem requires at most $\mathcal{O}(\gamma log(\frac{1}{\epsilon}))$ iterations to reach within the $\epsilon$ neighborhood of the primal optimal solutions $\mathbf{p}^*$ and $\mathbf{q}^*$ \cite{boyd2004convex}. The stability of the dual update dynamical systems in (\ref{Network_Opt}b) and (\ref{Network_Opt}d) depends on the weight matrix $W$ and the analysis of it was provided in \cite{Xiao_L}. Given the \emph{optimal} active and reactive powers and the nominal grid voltage, the low-level current reference for $i^{th}$ IBR can be generated based as follows:
\begin{equation}\label{current_ref}
    \mathbf{i}_{ref} = \frac{\mathbf{v}_g}{\norm{\mathbf{v}_g}_2^2}p_i^{N}+J\frac{\mathbf{v}_g}{\norm{\mathbf{v}_g}_2^2}q_i^{N},
\end{equation}
where $\mathbf{i}_{ref} \in \mathbb{R}^2$. Let $p_i^{N}, q_i^{N}$ be the exact solutions to the unconstrained optimization problems in (\ref{Network_Opt}a,\ref{Network_Opt}c), then 
\begin{equation}\label{closed_current_ref}
  \mathbf{i}_{ref} = -\frac{1}{\beta_i(\tilde{r}_{g_i})}\bigg(\frac{\mathbf{v}_g}{\norm{\mathbf{v}_g}_2^2}\lambda_i^N + J \frac{\mathbf{v}_g}{\norm{\mathbf{v}_g}_2^2}\nu_i^N\bigg),   
\end{equation}
from which it is seen that that the reference $\mathbf{i}_{ref}$ to the low-level controller decreases monotonically with $\beta_i(\tilde{r}_{g_i})$. Figure \ref{Control_Architecture} shows the overall control architecture.
\begin{figure}[t!] 
\centerline{\includegraphics[width=2.2 in]{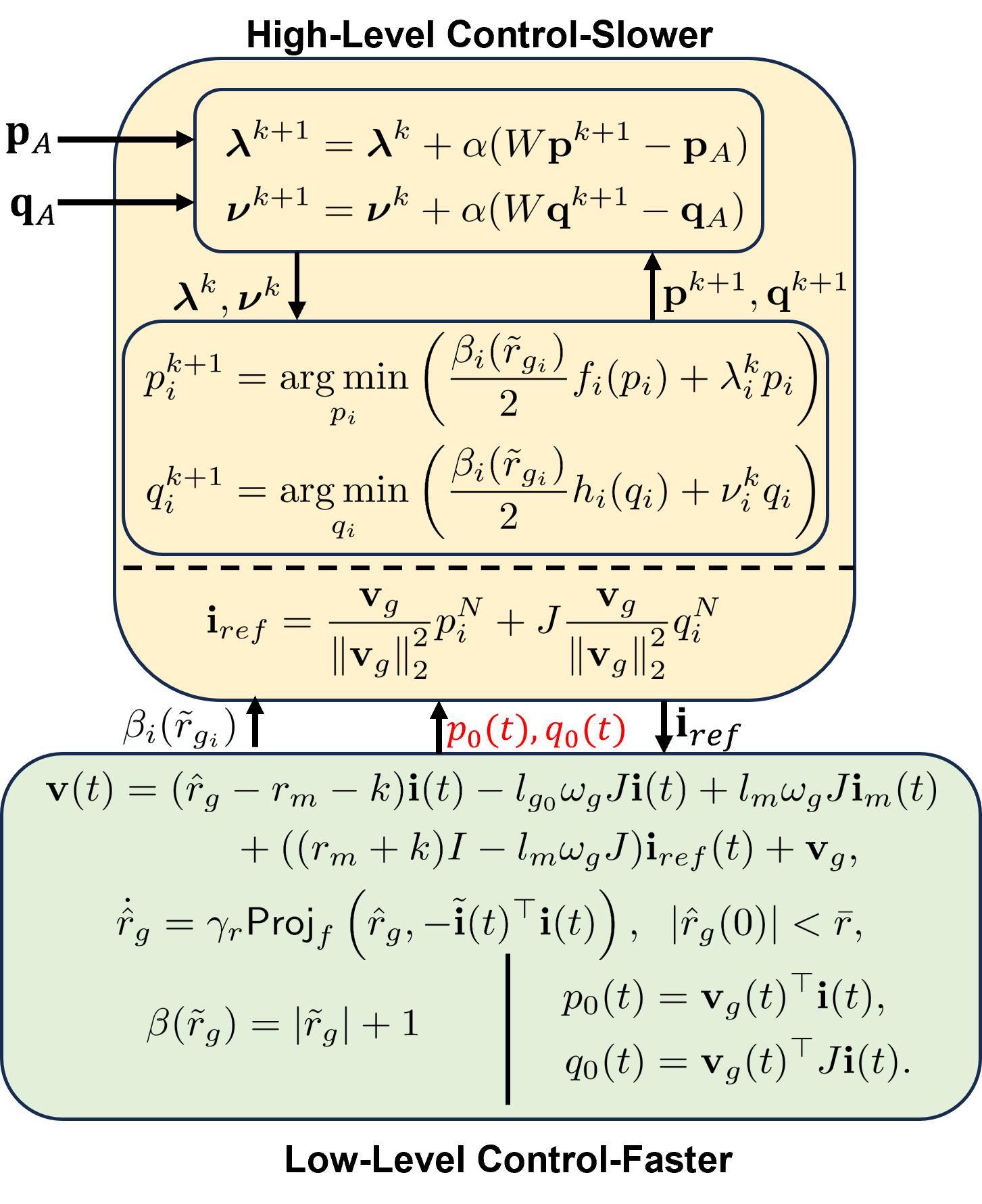}}
\caption{Control Architecture}
\label{Control_Architecture}
\end{figure}

\section{Numerical Simulation}\label{Sim}
The simulation is performed on a system with 3-IBRs. We consider a scenario where there is a fault in Inverter-3 at $t=0.2s$. Also, a grid voltage swell (as disturbance) of $10\%$ the nominal voltage is induced at $t=0.2s$. The system parameters used in the simulation are the transmission line resistance $r_g=0.027\Omega$, inductance $l_g=0.0367$\textsf{H}, the nominal grid voltage $v_g = 392$\textsf{V} (LN rms), the grid frequency $w_g=60$\textsf{Hz}. The optimization parameters chosen are $\alpha = 0.1$, the optimization weights $\beta_i=1$ under normal operation, and $\beta_i=10^4$ under faulty conditions. The adjacency matrix, the degree matrix, and the Laplacian based weight matrix are determined as follows: $$A=\begin{bmatrix} 0 & 1 &0\\1 & 0 & 1\\0 &1 &0 \end{bmatrix}, D=\begin{bmatrix} 1 & 0 &0\\0 & 2 & 0\\0 &0 &1 \end{bmatrix}, W=\begin{bmatrix} 0.666 & 0.333 & 0\\0.333 & 0.333 & 0.333\\0 &0.333 &0.666 \end{bmatrix}$$
it can be seen that the matrix $W$ is doubly stochastic. The results of the designed \emph{decentralized splitter} controller are compared with the \emph{adaptive splitter} based controller proposed in \cite{AMELI}. The bottom half of Figure \ref{fig_example2.jpg} (c), (d) shows the active power distribution in the adaptive splitter-based and decentralized control-based design. It can be seen that when there is a fault, the split mechanism in the adaptive splitter-based design slowly adjusts to balance the aggregated active power. This is because the low-level control in this approach is a function of the high-level distribution. Whereas, in the decentralized-based approach, it can be seen that the distribution almost instantaneously adjusts to balance the active power references to the low-level control, this is because of the changes in the optimization penalty weight $\beta_i(\tilde{r}_{g_i})$. Moreover, the low-level control designed asymptotically tracks the high-level generated power reference.

\begin{figure}[t!] 
\centerline{\includegraphics[width=3.4 in]{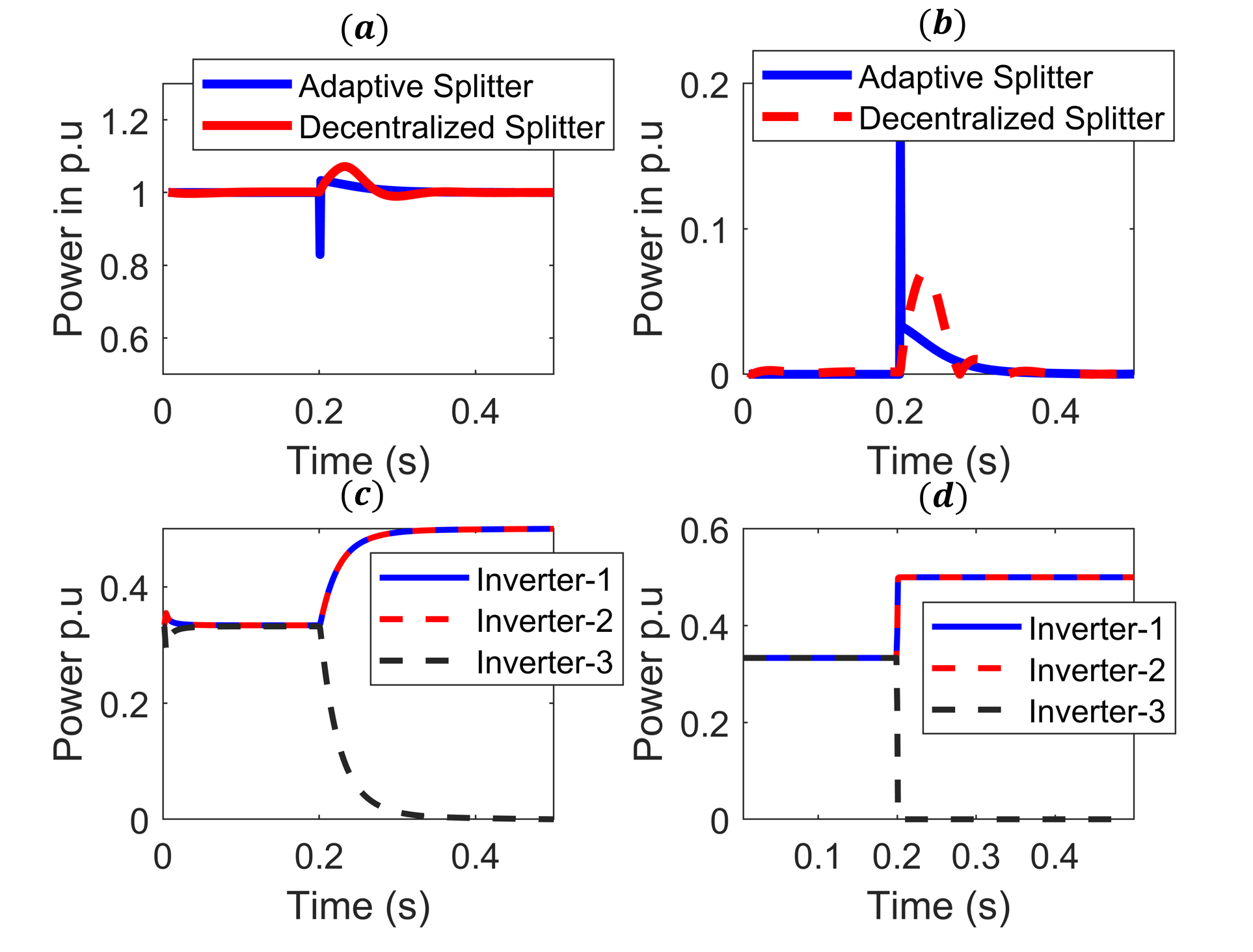}}
\caption{Active power tracking and tracking error (top) and power split (reference generated) for the adaptive splitter-based control design (left) and decentralized control design (right).}
\label{fig_example2.jpg}
\end{figure}

\begin{figure}[t!] 
\centerline{\includegraphics[width=3.5 in]{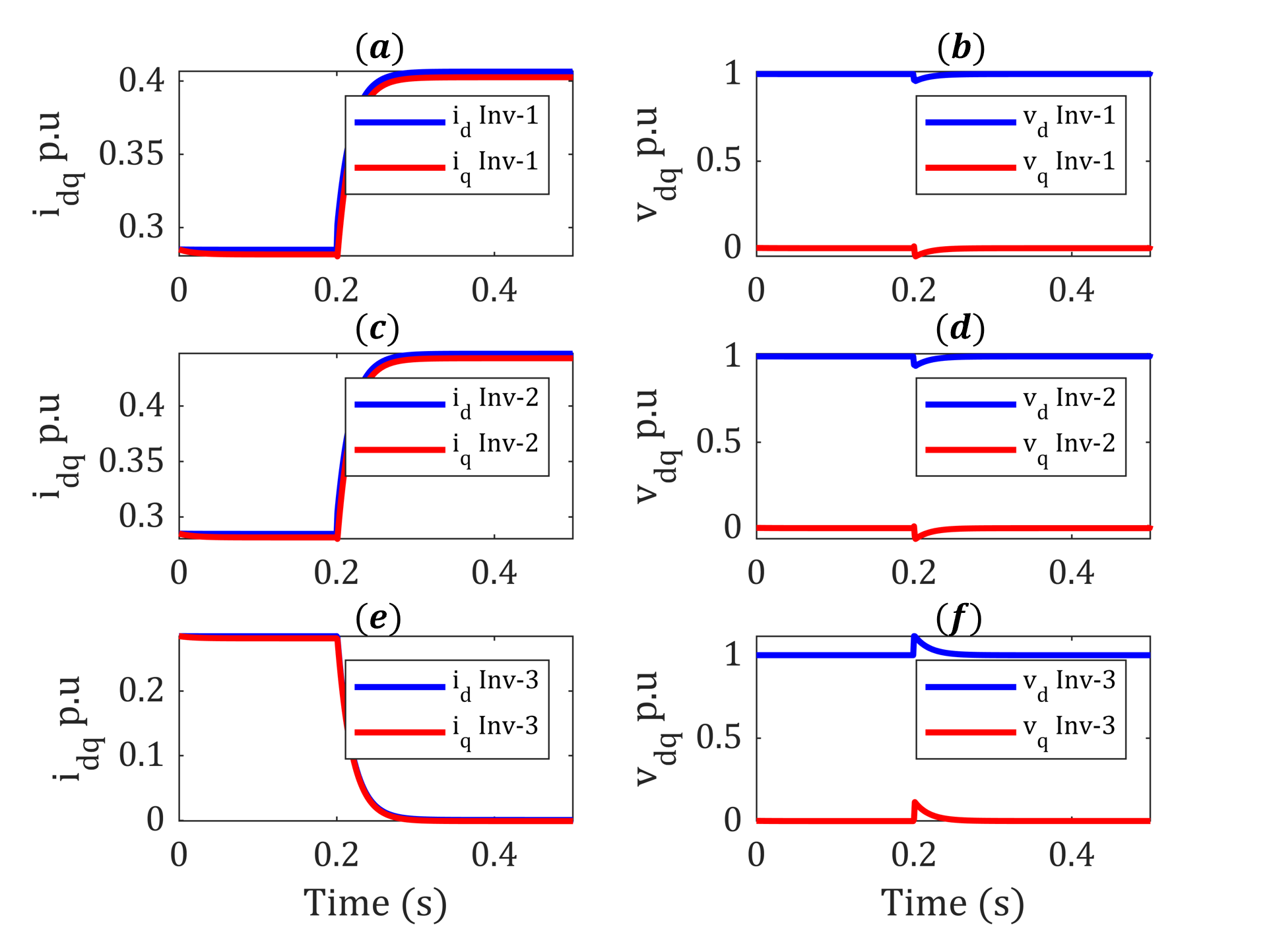}}
\caption{Currents $\mathbf{i}$ and Voltages $\mathbf{v}$ of the individual inverters before and after the fault}
\label{fig_example.jpg}
\end{figure}

The aggregated active power tracking shown in the top half of Figure-\ref{fig_example2.jpg} (a), (b) compares the high-level active power distribution between the adaptive control-based splitter design and the decentralized optimization-based splitter design. It can be seen that when the fault and the grid voltage swell occur at $t=0.2s$, the decentralized optimization-based controller provides a better tracking response. The magnitude of active power deviation from the nominal $1$p.u value is significant while using the adaptive control-based power distribution mechanism. The result shows a significant improvement in the context of regulating active power at nominal value from the existing design technique in the literature. 

Finally, the currents ($\mathbf{i}$) and the input voltages ($\mathbf{v}$) for the three inverters are shown in Figure \ref{fig_example.jpg}. It can be seen from the subplot (e) that the d and q-axis currents of Inverter-3 go to zero after the fault. Similarly, the impact of the fault can also be seen on the d and q-axis voltages of the Inverter-3 from subplot (f). Moreover, the current sharing between the IBRs is split between the remaining two IBRs with healthy transmission lines which can be seen from subplots (a) and (c).

Thus, the results demonstrate the impact of the designed high-level decentralized optimization-based control. The comparison with the adaptive control-based power splitter further solidifies the control effectiveness of the designed high-level controller.  

\section{Conclusion and Future Work}\label{Conc}

This paper presents a decentralized high-level control for the aggregated active power tracking in the IBRs under complete loss of inverter due to the faults in the transmission line. The proposed algorithm adaptively splits the active power taking into consideration the condition of the IBR lines. A parameter estimator is developed which determines the healthiness of the transmission line and updates the high-level algorithm on the transmission line's status. The results show an improvement in the active power tracking response compared to the previously proposed method. The developed decentralized algorithm relies on the established concept of duality. The future extensions of this work will focus on the stability of the developed decentralized control. Investigating the transmission limitations and their impact on power-sharing will also be reported in future extensions of this work.

\bibliographystyle{IEEEtran}
\bibliography{myreferences}

\end{document}